\documentclass[12pt]{amsart}
\usepackage{amssymb,amsmath,amsfonts,latexsym}
\usepackage{bm}
\usepackage[all,cmtip]{xy}
\usepackage{amscd}

\newcommand{\bea}{\begin{eqnarray}}
\newcommand{\eea}{\end{eqnarray}}

\def\textmatrix#1&#2\\#3&#4\\{\bigl({#1 \atop #3}\ {#2 \atop #4}\bigr)}
\def\dispmatrix#1&#2\\#3&#4\\{\left({#1 \atop #3}\ {#2 \atop #4}\right)}
\newcommand{\be}{\begin{equation}}
\newcommand{\ee}{\end{equation}}
\newcommand{\ben}{\begin{eqnarray*}}
\newcommand{\een}{\end{eqnarray*}}

\newcommand{\bi}{\begin{itemize}}
\newcommand{\ei}{\end{itemize}}

\newtheorem{Theorem}{\sc Theorem}[section]
\newtheorem{Lemma}[Theorem]{\sc Lemma}
\newtheorem{Proposition}[Theorem]{\sc Proposition}
\newtheorem{Corollary}[Theorem]{\sc Corollary}
\theoremstyle{definition}

\theoremstyle{plain}

\newtheorem{thm}{Theorem}[section]

\theoremstyle{definition}

\numberwithin{equation}{section}

\let\phi=\varphi
\begin{document}
\title[Improved Discrete Dual $p$-Hardy and Weighted Discrete $p$- Birman Inequalities]
{Improved Discrete Dual $p$-Hardy and Weighted Discrete $p$-Birman Inequalities}

\author[Das]{Bikram Das}
\address{Indian Institute of Science, Bengaluru-560012, Karnataka, India.}
\email{dasb23113@gmail.com}


\begin{abstract}
In this paper, we establish a new version of one dimensional generalized discrete dual $p$-Hardy inequality with a shift. Using this generalized discrete dual p-Hardy inequality, we obtain  improvements of two discrete dual $p$-Hardy inequalities. To be specific, for $p>1$ and $A\in C_c(\mathbb{N}_{0})$ satisfying $A_{0}=A_{1}=0$, we first improve the discrete dual p-Hardy inequality
\begin{align*}
&\displaystyle\sum_{n=2}^{\infty}(n-1)^{p}| A_{n}-A_{n-1}|^{p}\geq\frac{1}{p^{p}}\displaystyle\sum_{n=2}^{\infty}|A_{n}|^{p},
\end{align*} where the associate constant term is sharp. Subsequently, we improve its power-type weighted discrete dual p-Hardy extension \\
\begin{align*}
&\displaystyle\sum_{n=2}^{\infty}(n-1)^{\alpha}|A_{n}-A_{n-1}|^{p}\geq\Big(\frac{\alpha+1-p}{p}\Big)^{p}
\displaystyle\sum_{n=2}^{\infty}\frac{|A_{n}|^{p}}{n^{p-\alpha}}
\end{align*} for $p-1<\alpha\leq p$, where the associated constant term is also sharp. We also establish a discrete $p$- Birman inequality with  power weights.  Furthermore, we establish a multivariable dual $p$-Hardy inequality with a sharp constant. The proof proceeds by first establishing the inequality for two variables and then extending the argument to multiple variables, while preserving the sharpness of the constant.
\end{abstract}
\maketitle

\section{Introduction}{\label{secint}}
The development of Hardy-type inequalities has its roots in mathematical work from the beginning of the twentieth century. On dated 21st June, 1921, G.H. Hardy received a letter by E. Landau containing a proof of the following famous discrete Hardy inequality:\\
Let $p>1$ be a real number and $a=\{a_n\}$ be a sequence of real numbers, then the classical discrete $p$-Hardy inequality ( \cite{GHYLITTLE}, Theorem 326) in one dimension asserts that
\begin{align}{\label{pDHI}}
&\displaystyle\sum_{n=1}^{\infty}a_{n}^{p}\geq\Big(\frac{p-1}{p}\Big)^{p}\displaystyle\sum_{n=1}^{\infty}\Big(\frac{a_{1}+a_{2}\cdots+a_{n}}{n}\Big)^{p},
\end{align} where the associated constant term $\Big(\frac{p-1}{p}\Big)^{p}$ is sharp and equality holds if and only if $a_n=0$ for all $n\in \mathbb{N}$. Since the inequality was first stated by Hardy, it is commonly referred to as Hardy inequality \cite{AKR}. The origin of the inequality is closely related to Hardy's investigations of Hilbert's theorem. In 1918, Hardy was concerned with finding a short and elegant proof of Hilbert's theorem on the convergence of double series \cite{GHYNOTE}. Although the inequality was not formulated explicitly in that work, the essential reasoning required for its proof can already be found in Hardy's argument. The development of the inequality was not limited to the contributions of G. H. Hardy himself (\cite{GHYNOTE,GHYNOTE251}).
  Several prominent mathematicians, including E. Landau, G. $P\acute{o}lya$, M. Riesz, and I. Schur, made important contributions to its subsequent development. In particular, Landau's (\cite{ELU, LANDAU}) work  played a significant role in establishing inequality (\ref{pDHI}). Owing to these contributions, the result is also referred to in the literature as the Hardy-Landau inequality.
 Let $A=\{A_n\}\in C_{c}(\mathbb{N}_0)$ with $A_0=0$, where $C_{c}(\mathbb{N}_0)$ is a space of all finitely supported functions on $\mathbb{N}_0:=\{0,1,2,...\}$, then the inequality (\ref{pDHI}) takes the equivalent form
\begin{align}{\label{EFpDHI}}
&\displaystyle\sum_{n=1}^{\infty}|A_{n}-A_{n-1}|^{p}\geq\big(\frac{p-1}{p}\big)^{p}\displaystyle\sum_{n=1}^{\infty}\frac{|A_n|^{p}}{n^{p}}.
\end{align}
 Although Hardy's inequality has been extensively studied in the continuous setting (see \cite{BEL}, \cite{OK}, \cite{RFR}, \cite{GM}, and the references therein), its discrete analogue remains comparatively less explored. The discrete setting presents several additional challenges. One of the principal difficulties arises from the breakdown of the usual calculus in the discrete setting. \\ A remarkable breakthrough in this direction was achieved by Keller, Pinchover, and Pogorzelski in 2018 \cite{MKR} (see also \cite{MKRGRAPH}). They showed that, although the classical Hardy constant $\frac{1}{4}$ is sharp, the corresponding weight $\frac{1}{4n^{2}}$ is not optimal. In fact, they proved the following improved inequality
\begin{align}{\label{IMHI2}}
\displaystyle\sum_{n=1}^{\infty}| A_{n}-A_{n-1}|^{2}&\geq\displaystyle\sum_{n=1}^{\infty} w_n^{KPP}|A_n|^{2},
\end{align} where weight sequence $w_n^{KPP}=2-\Big({1-\frac{1}{n}}\Big)^{1/2}-\Big({1+\frac{1}{n}}\Big)^{1/2}>\frac{1}{4n^2}=w_n^{H}$, $n\in \mathbb{N}$.\\
The classical discrete Hardy inequality has been the subject of extensive research in recent years, particularly with regard to its improvement and generalization. This ongoing development has led to several improved forms and broader extensions. Notably, \cite{DKALFS} obtained generalized versions of the improved discrete Hardy inequality (\ref{IMHI2}). Subsequently, the results were further advanced in \cite{DASMAN}, where (\ref{IMHI2}) was extended in a more general setting, encompassing the earlier generalization presented in \cite{DKALFS}. Further contributions in this area were made by  Krej\v{c}i\v{r}\'{i}k and \v{S}tampach \cite{DKK}, and Gerhat et al. \cite{BGDKFS}, where different forms of improved inequalities were derived through elementary arguments and suitable factorization techniques. Related developments can also be found in \cite{HUANGYE}, which considers improvements of the Hardy-Rellich inequality, particularly in the higher-order setting.\\A further significant extension was obtained by Fischer, Keller, and Pogorzelski \cite{FFR}. Their result generalizes the previously established inequality (\ref{IMHI2}) for every real parameter $p>1$. More precisely, they proved that
\begin{align}{\label{one Improve}}
\displaystyle\sum_{n=1}^{\infty}|A_{n}-A_{n-1}|^{p}&\geq\displaystyle\sum_{n=1}^{\infty} w_n^{FKP}(p)|A_n|^{p},
\end{align}where for each $n\in \mathbb{N}$, the improved weight sequence $w_n^{FKP}(p)$ is defined by
\begin{center}
$w_n^{FKP}(p)=\big(1-(1-\frac{1}{n})^{\frac{p-1}{p}}\big)^{p-1}-\big((1+\frac{1}{n})^{\frac{p-1}{p}}-1\big)^{p-1}>\big(\frac{p-1}{p}\big)^{p}\frac{1}{n^{p}}$.
\end{center}
More recently, the authors of \cite{HUANGYE} established a generalized improved discrete Hardy inequality. In particular, they obtained the following identity in \cite{HUANGYE}:
\begin{align}{\label{GIMHI}}
\displaystyle\sum_{n=1}^{\infty}V_{n}|\nabla A_{n}|^{2}+\displaystyle\sum_{n=1}^{\infty}\frac{\operatorname{div}\left(V(\nabla\mu)\right)_{n}}{\mu_{n}}|A_n|^{2}=\displaystyle\sum_{n=2}^{\infty}V_{n}\Big|\sqrt{\frac{\mu_{n-1}}{\mu_{n}}}A_{n}-\sqrt{\frac{\mu_{n}}{\mu_{n-1}}}A_{n-1}\Big|^{2},
\end{align} where, $\mu=\{\mu_n\}$ are strictly positive sequence of real numbers, $V_{n}\geq0$, and the discrete gradient and divergence are defined respectively by $\nabla\mu_{n}:=\mu_{n}-\mu_{n-1}$ and $\operatorname{div}h_{n}:=h_{n+1}-h_{n}$, for $n\geq1$.\\
A recent contribution to this area was presented by \v{S}tampach and Waclawek \cite{FS26}, who further developed the method introduced in \cite{HUANGYE} and extended it to the general range $1<p<\infty$.\\
 More precisely, suppose that $p>1$, $V_{n}\geq0$ for every $n\in\mathbb{N}$, and let $\mu$ be a positive non-decreasing sequence satisfying $\mu_{n+1}\geq\mu_{n}>0$, $n\in\mathbb{N}$ with $\mu_{0}=0$. Then, for every complex-valued sequence $A\in C_c(\mathbb{N}_{0})$ satisfying $A_{0}=0$, they established the estimate
\begin{align}\label{SW26}
\sum_{n=1}^{\infty}V_{n}|\nabla A_{n}|^p\geq-\sum_{n=1}^{\infty}\frac{\operatorname{div}\left(V(\nabla\mu)^{p-1}\right)_{n}}{\mu_{n}^{p-1}}|A_{n}|^p.
\end{align}
This result provides a unified formulation of the corresponding weighted inequality for all $p>1$. For $p=2$, the inequality (\ref{SW26}) reduces to (\ref{GIMHI}).
\subsection{Improved Discrete Hardy Inequalities with Power Weights}
The weighted discrete Hardy inequality was introduced in the work of Copson \cite{ECN2} as a weighted counterpart of the classical inequality (\ref{EFpDHI}). In particular, by taking $\lambda_{n}=1$ for every $n\in\mathbb{N}$ and $p-c=\alpha$, where $1<c\leq p$, in Theorem 1.1 of \cite{ECN2}, we have the following special case of Copson inequality in its equivalent form, given by
\begin{align}{\label{COPIPAR}}
\displaystyle\sum_{n=1}^{\infty}{n}^{\alpha}|A_{n}-A_{n-1}|^{p}&\geq\Big( \frac{p-\alpha-1}{p}\Big)^{p}\sum_{n=1}^{\infty}\frac{|A_{n}|^{p}}{n^p}{n}^{\alpha},~~\alpha\in [0, p-1),
\end{align}where $\{A_n\}$ is a sequence of complex numbers and the associated constant term is sharp.
 This particular case was considered earlier by Hardy and Littlewood \cite{HL1927}, who obtained a power-type weighted inequality but did not identify the exact optimal constant. The sharp constant was subsequently established by Copson \cite{ECN2}; see also Bennett \cite[Corollary 3]{BENN} for a related formulation. In a recent work,  Gupta \cite{SGA} investigated the corresponding power-type discrete Hardy inequality (\ref{COPIPAR}) for case $p=2$ and achieved an improvement for $\alpha\in \{0\}\cup[\frac{1}{3}, 1)$. It is worth noting that the inequality (\ref{COPIPAR}) for case $p=2$ remains valid for $\alpha\geq0$ and was established in \cite{HUANGYE}. More recently, \v{S}tampach and Waclawek \cite{FS26} proved that the inequality (\ref{COPIPAR}) holds for all $\alpha<0$, provided that a slightly smaller weight is used on the right-hand side. In fact, the authors in \cite{FS26} proved the inequality  as below:\\
If $p>1$ and $A\in C_c(\mathbb{N}_{0})$ with convention $A_{0}=0$, then for all $\alpha<0$, we have
\begin{align}{\label{ECOPIPAR}}
\displaystyle\sum_{n=1}^{\infty}n^{\alpha}|A_{n}-A_{n-1}|^{p}&\geq\Big( \frac{p-\alpha-1}{p}\Big)^{p}\sum_{n=1}^{\infty}(n+1)^{\alpha-p}|A_{n}|^{p},
\end{align} where the attached constant term $\Big( \frac{p-\alpha-1}{p}\Big)^{p}$ is sharp.
\subsection{ Dual p-Hardy Inequalities and Their Improvements}
It is worth noting that relatively few studies have been devoted to improvements of the discrete dual Hardy inequality. To place this problem in the paper, we briefly recall its classical development. The study of dual Hardy inequalities originated with the work of Hardy \cite{GHYNOTE1919} (see also \cite{GHYNOTE254}) in 1919.\\  For any $p>1$ and any non-negative sequence $\{a_n\}$, the following discrete dual $p$-Hardy inequality holds:
\begin{align}{\label{VHARDYI}}
&\displaystyle\sum_{n=1}^{\infty}\Big|\displaystyle\sum_{k=n}^{\infty}\frac{a_{k}}{k}\Big|^{p}\leq p^{p}\displaystyle\sum_{n=1}^{\infty}|a_{n}|^{p},~~ i.e~ \displaystyle\sum_{n=1}^{\infty}\Big|\displaystyle\sum_{k=n}^{\infty}a_{k}\Big|^{p}\leq p^{p}\displaystyle\sum_{n=1}^{\infty}n^{p}|a_{n}|^{p}
\end{align} where the associated constant term `$p^{p}$' is sharp, and equality holds if and only if $a_{n}=0$ for all $n\in\mathbb{N}$. Hardy \cite{GHYNOTE1919} first established the dual form (\ref{VHARDYI}) of (\ref{pDHI}) in 1919 for $p=2$, and Copson \cite{ECN} subsequently generalized it to $p>1$. Consequently, this inequality is sometimes referred to as the Hardy-Copson inequality ( see \cite{AKR}). The duality between (\ref{VHARDYI}) and (\ref{pDHI}) was first pointed out by Hardy \cite{GHYNOTE254}. The \emph{dual} inequality (\ref{VHARDYI}) is also called as discrete \emph{variant} Hardy's inequality. By choosing $A_n=\displaystyle\sum_{k=n}^{\infty}a_{k}$, inequality (\ref{VHARDYI}) can be equivalently written as
\begin{align}{\label{VHARDYIP2}}
&\displaystyle\sum_{n=2}^{\infty}(n-1)^{p}|A_{n}-A_{n-1}|^{p}\geq\frac{1}{p^{p}}\displaystyle\sum_{n=2}^{\infty}|A_{n}|^{2},
\end{align}where $\{A_n\}$ is a sequence of complex numbers such that $A_{0}=A_{1}=0$.
The inequality (\ref{VHARDYIP2}) was subsequently generalized to a broader class of power weights. More precisely, for $\alpha\in(p-1, p]$, the discrete dual Hardy inequality takes the form
\begin{align}{\label{DHI}}
\displaystyle\sum_{n=2}^{\infty}(n-1)^{\alpha}|A_{n}-A_{n-1}|^{p}&\geq\Big(\frac{1+\alpha-p}{p}\Big)^{p}\displaystyle\sum_{n=2}^{\infty}\frac{|A_{n}|^{2}}{n^2}n^{\alpha},~~ A_{0}=A_{1}=0,
\end{align}where the associated constant term $\Big(\frac{1+\alpha-p}{p}\Big)^{p}$ is sharp. Such weighted extensions were investigated by Copson (\cite{ECN, ECN2}), who established the corresponding inequalities together with their optimal constants.\\The improvement of the dual Hardy inequality (\ref{VHARDYIP2}) was first established in \cite{DASMAN3} for the case $p=2$. In fact, they established  the following inequality:
\begin{align}{\label{VHRDYalpha2}}
&\displaystyle\sum_{n=2}^{\infty}(n-1)^{2}|A_{n}-A_{n-1}|^{2}\geq\displaystyle\sum_{n=2}^{\infty}\beta_{n}|A_{n}|^{2}, A_{0}=A_{1}=0,
\end{align}where the weight sequence $\beta_n$ for $n\geq 2$ is given by
\begin{align*}
\beta_{n}&=n^{2}\Big[1+\Big(1-\frac{1}{n}\Big)^{2}-\Big(1+\frac{1}{n}\Big)^{-\frac{1}{2}}-\Big(1-\frac{1}{n}\Big)^{\frac{3}{2}}\Big].
\end{align*}  Moreover, $\beta_{n}>\frac{1}{4}$ for $n\geq 2$ . Subsequently, the authors in \cite{DASMANTON} obtained a corresponding improvement for the weighted dual Hardy inequality (\ref{DHI}), again in the case $p=2$. In fact the authors of \cite{DASMANTON} established the following improved dual Hardy inequality
\begin{align}{\label{BETAnALPHA}}
\displaystyle\sum_{n=2}^{\infty}(n-1)^{\alpha}|A_{n}-A_{n-1}|^{2}&\geq\displaystyle\sum_{n=2}^{\infty}\beta_{n}(\alpha)|A_{n}|^{2}, ~~A_0=A_1=0,
\end{align}
where $\alpha\in(1, 2]$, and the improved weight sequence  $\beta_{n}(\alpha)$ is given by
\begin{align*}
\beta_{n}(\alpha)& =n^{\alpha}\Big[1+\Big(1-\frac{1}{n}\Big)^{\alpha}-\Big(1-\frac{1}{n}\Big)^{\frac{1+\alpha}{2}}
-\Big(1+\frac{1}{n}\Big)^{\frac{1-\alpha}{2}}\Big], ~n\geq 2.
\end{align*}.
 Furthermore, the following estimate holds for all $n\ge2$:
\begin{align*}
&\beta_n(\alpha)>\frac{(\alpha-1)^{2}}{4}n^{\alpha-2},~~\alpha\in(1, 2].
\end{align*}In addition, it is shown that the improved weight sequence $\beta_{n}(\alpha)$ is critical.While improvements of Hardy inequalities have been extensively studied for $p>1$, the corresponding improvement problem for dual Hardy inequalities appears to remain unexplored. The present work is devoted to this problem. We initiate the study of improved dual Hardy inequalities and obtain corresponding improvements for every $p>1$. Our investigation is therefore motivated by the following question:\\

\textsf{Q(a) Do the discrete dual Hardy inequality (\ref{VHARDYIP2}) and its power-weighted version (\ref{DHI}) admit further improvements for every $p>1$ ?}\\
\subsection{Discrete Birman Inequalities}
In 1961, Birman \cite{BM} established a higher-order extension of Hardy's inequality by considering derivatives of arbitrary order in the case $p=2$.
%In 1961, Birman \cite{BM} generalized the case $p=2$ of continuous version of the classical discrete Hardy inequality (\ref{EFpDHI}) to derivatives of arbitrary order.
More precisely, for $\phi\in C_{0}^{\infty}\mathbb{(R_{+})}$ and $l\in\mathbb{N}$,
\begin{align}{\label{CBMIp2}}
&\int_{0}^{\infty}|\phi^{(l)}(x)|^{2}dx\geq\Big(\frac{1}{2}\Big)^{2}_{l}\int_{0}^{\infty}\frac{|\phi(x)|^{2}}{x^{2l}}dx,
\end{align}where $(a)_{l}= a(a + 1)\dots(a + l-1)$ denotes Pochhammer symbol.
%A proof of (\ref{CBMIp2}) can be found in [\cite{GM}, pp. 83.84]; see also the recent paper\cite{FGLWR} for other proofs 	
%and generalizations. The particular case $p = 2 $ of (\ref{CBMIp2}) was discovered earlier by Rellich \cite{RFH}, 	
%even in higher dimensions, and is commonly referred to as the Rellich inequality.
The proof of this inequality and its various extensions have been considered in (\cite{ FGLWR, GM}); the case $l=2$ is the well-known Rellich inequality \cite{RFH}.
If $p>1$, $l\in\mathbb{N}$, then for $\phi\in C_{0}^{\infty}\mathbb{(R_{+})}$, the continious $p$-Birman inequality is given as follows:
\begin{align}{\label{CbirmanP}}
&\int_{0}^{\infty}|\phi^{(l)}(x)|^{p}dx\geq\Big(\frac{p-1}{p}\Big)^{p}_{l}\int_{0}^{\infty}\frac{|\phi(x)|^{p}}{x^{pl}}dx.
\end{align}
To the best of our knowledge, no explicit reference for the general form of (\ref{CbirmanP}) is available in the literature. The case $l=2$, corresponding to the $p$-Rellich inequality, can be found in \cite{Owen}. The continuous $p$-Birman inequality can also be obtained iteratively from the weighted form of (\ref{CBMIp2}) established in \cite[Theorem 2]{KMP} and \cite[Theorem 330]{GHYLITTLE52}. The inequality (\ref{CbirmanP}) is studied in \cite{FS24} for particular case. More recently, \v{S}tampach and Waclawek \cite{FS26} gave a detailed proof of (\ref{CbirmanP}) and surprisingly,  established the discrete $p$-Birman inequality with the sharp constant.\\
If $p>1$,  $A\in C_c(\mathbb{N}_0)$ satisfy  $A_{n}=0$ for $n<l$, where $l\in\mathbb{N}$, then for all $n\geq l$, $n\in\mathbb{N}_0$, the following inequality holds.
\begin{align}{\label{PBirman}}
&\displaystyle\sum_{n=l}^{\infty}|\nabla^{l} A_{n}|^{p}\geq\Big(\frac{p-1}{p}\Big)_{l}^{p}\displaystyle\sum_{n=l}^{\infty}\frac{|A_{n}|^{p}}{n^{lp}}.
\end{align} The above inequality (\ref{PBirman}) was first deduced for p = 2 by Huang and Ye in \cite{HUANGYE}. For $l=1$, inequality (\ref{PBirman}) reduces to the classical Hardy inequality (\ref{EFpDHI}), which admits the power-type weighted extension (\ref{COPIPAR}). This naturally raises the question of whether a similar power-weighted extension of the $p$-Birman inequality exists in the discrete setting for every $l\in\mathbb{N}$. As far as we are aware,, this problem has not yet been addressed. Thus, we pose the following question:\\

\textsf{Q(b) Does the discrete $p$-Birman inequality (\ref{PBirman}) admit a power-weighted extension?}

\subsection{Hardy Inequalities in Several Variables}
The extension of discrete Hardy inequalities from one variable to several variables has been extensively studied. Pachpatte \cite{PACH} initiated the study of multivariable discrete $p$-Hardy inequalities, and Salem et al.\cite{SZR} subsequently established sharp forms for double sequences and, more generally, for $r$-fold sequences. To be specific, for $r$ variables $m_{1},m_{2},\ldots m_{r}$ , the multivariable discrete Hardy inequality takes the form \\
\begin{align}{\label{r fold inequality}}
&\displaystyle\sum_{m_{1}=1}^{\infty}\displaystyle\sum_{m_{1}=1}^{\infty}\ldots\displaystyle\sum_{m_{r}=1}^{\infty}\Big|\displaystyle\sum_{i_{1}=1}^{m_{1}}\displaystyle\sum_{i_{2}=2}^{m_{2}}\ldots\displaystyle\sum_{i_{r}=1}^{m_{r}}a_{i_{1}i_{2}\ldots i_{r}}\Big|^{p}\nonumber\\
&\leq (\frac{p}{p-1})^{pr}\displaystyle\sum_{m_{1}=1}^{\infty}\displaystyle\sum_{m_{2}=1}^{\infty}\ldots\displaystyle\sum_{m_{r}=1}^{\infty}|a_{m_{1}m_{2}\ldots m_{r}}|^{p},
\end{align}where the associated constant term $(\frac{p}{p-1})^{pr}$ is sharp. The corresponding theory for discrete $p$-dual Hardy inequalities remains considerably less developed. In particular, while  several variables extensions of the discrete $p$-Hardy inequality are well established, the question of whether an analogous extension exists for its dual counterpart remains open. This naturally motivates us to investigate the several variables structure of the discrete $p$-dual Hardy inequality.  This leads to the following question:\\

\textsf{Q(c) Can the discrete dual $p$-Hardy inequality (\ref{VHARDYI}) be extended to two variable case and, more generally, to the multivariable setting?}\\

The primary aim of this paper is to provide answers to the questions Q(a)-Q(c) formulated above. Our approach begins with the introduction of the $m$-shifting discrete gradient and divergence, viewed as difference operators on $\mathbb{N}$.  Using these operators, we derive a generalized $m$-shifting improved discrete dual $p$-Hardy inequality. Appropriate choices of $\mu_{n}$ yield weight sequences that further improve (\ref{VHARDYIP2}) and its power-weighted counterpart (\ref{DHI}) for all $p>1$. Consequently, Q(a) receives a positive answer. We then turn to Q(b), where a power-type weighted discrete $p$-Birman inequality is obtained. For $p=2$ and $l=2$, $p$-Birman inequality subsequently leads to a weighted discrete Rellich inequality. Whether the constant associated with the corresponding weighted Birman inequality is optimal remains a open question. Finally, addressing Q(c), we establish a multidimensional dual Hardy inequality with a sharp constant. The proof is first established in the two-variable case and then extended to the general $k$-variable setting, thus addressing Q(c).\\

The paper is organized as follows. Section 2 establishes a general shifted discrete dual $p$-Hardy inequality. Section 3 develops further improvements of the discrete dual $p$-Hardy inequality and its power-weighted form through suitable choices of $\mu_n$, $n\in\mathbb N$. Section 4 is devoted to a weighted discrete $p$-Birman inequality and its consequence, a weighted discrete Rellich inequality. Finally, Section 5 establishes multiple variables discrete dual $p$ Hardy inequality with a sharp constant.

\section{Generalized Shifted Improved Dual $p$-Hardy Inequality}
We define the $m$-shifting discrete gradient and divergence as difference operators acting on the space of sequence $C_{c}(\mathbb{N}_{0})$. For $m\in\mathbb{N}$, these operators are given by
\begin{align}
\nabla_{m}\phi_{n} &=
 \left\{
\begin{array}{lll}
    \phi_{n} & \quad \mbox{if~~} n=1,2,\ldots m,\\
    %\frac {2\log N-\sqrt\lambda_{n}\log n}{\log N} & \quad \mbox{if~~} N\leq n\leq N^{2}\\
    \phi_{n}-\phi_{n-m} & \quad \mbox{if~~} n> m,
\end{array}
\right.
\end{align} and for $n\geq m$, $\operatorname{div}_{m}\phi_{n}=\phi_{n+m}-\phi_{n}$.\\
We first recall the following lemma, which will be used in the proof of the proposition below. This result is taken from Lemma 2.6 of \cite{RFR}, and we state it here without proof.
\begin{Lemma}\cite{RFR}{\label{lemsinq}}
Let $p>1$, $a\in\mathbb{C}$ and $t\in[0, 1]$. Then,
\begin{align*}
 &(1-t)^{p-1}(\mid a\mid^{p}-t)\leq\mid a-t\mid^{p}.
\end{align*}
\end{Lemma}
We begin with the following proposition.
\begin{Proposition}{\label{p-weightedgeneral}}
Let $p>1$ and $V_n\geq0$ for $n\in\mathbb N$. Suppose that ${\mu_n}$, $n\in\mathbb{N}$ is strictly positive sequence such that $\mu_{n}\geq\mu_{n+1}$ for $n\geq 1$ and  $\mu_{0}=0$. Then for every $A\in C_c(\mathbb{ N}_{0})$ satisfying $A_n=0$ for all $n\le m$, the following inequality holds.
\begin{align}
&\displaystyle\sum_{n=m+1}^{\infty}V_{n}|\nabla_{m}A_{n}|^{p}\geq\displaystyle\sum_{n=m+1}^{\infty}\frac{\operatorname{div}_{m}(V(-\nabla_{m}\mu)^{p-1})_{n}}{\mu_{n}^{p-1}}|A_{n}|^{p}.
\end{align}
\end{Proposition}
\begin{proof}
Let $p>1$ and $V_{n}\geq0$. By lemma \ref{lemsinq}, we have
\begin{align}{\label{Lelambda}}
&V_{n}|a-t|^{p}\geq V_{n}(1-t)^{p-1}(|a|^{p}-t).
\end{align}
 Since $\mu_{n}$ is non-increasing, we have $\mu_{n-m}\geq\mu_{n}$ for $n\geq m+1$, $n, m\in\mathbb{N}$. Hence for $n\geq m+1$, we may choose $t=\frac{\mu_{n}}{\mu_{n-m}}\in(0,1]$ and $a=\frac{\chi_{n-m}}{\chi_{n}}\in\mathbb{C}$, where $\chi_{n}\neq 0$, $m,n\in\mathbb{N}$. Substituting these choices into \eqref{Lelambda} and multiplying both sides by $|\chi_{n}|^{p}\mu_{n-m}^{p}$, we obtain,\\
\begin{align}{\label{Lelambdaput}}
&V_{n}|\mu_{n}\chi_{n}-\mu_{n-m}\chi_{n-m})|^{p}\geq V_{n}(\mu_{n-m}-\mu_{n})^{p-1}\big(|\chi_{n-m}|^{p}\mu_{n-m}-|\chi_{n}|^{p}\mu_{n}\big).
\end{align} We set $A_{n}=\mu_{n}\chi_{n}$, then we have,
\begin{align*}
&\displaystyle\sum_{n=m+1}^{\infty}|\nabla_{m}A_{n}|^{p}\\
=&\displaystyle\sum_{n=m+1}^{\infty}V_{n}|\mu_{n}\chi_{n}-\mu_{n-m}\chi_{n-m}|^{p}\\
\geq&\displaystyle\sum_{n=m+1}^{\infty}V_{n}(\mu_{n-m}-\mu_{n})^{p-1}\big(|\chi_{n-m}|^{p}\mu_{n-m}-|\chi_{n}|^{p}\mu_{n}\big)~~[by~~inequality~ (\ref{Lelambdaput})]\\
=&\displaystyle\sum_{n=m+1}^{\infty}V_{n}(\mu_{n-m}-\mu_{n})^{p-1}\frac{|A_{n-m}|^{p}}{\mu_{n-m}^{p-1}}-\displaystyle\sum_{n=m+1}^{\infty}V_{n}(\mu_{n-m}-\mu_{n})^{p-1}\frac{|A_{n}|^{p}}{\mu_{n}^{p-1}}\\
=&\displaystyle\sum_{n=1}^{\infty}V_{n+m}(\mu_{n}-\mu_{n+m})^{p-1}\frac{|A_{n}|^{p}}{\mu_{n}^{p-1}}-\displaystyle\sum_{n=m+1}^{\infty}V_{n}(\mu_{n-m}-\mu_{n})^{p-1}\frac{|A_{n}|^{p}}{\mu_{n}^{p-1}}\\
=&\displaystyle\sum_{n=m+1}^{\infty}\Big(V_{n+m}(\mu_{n}-\mu_{n+m})^{p-1}\frac{|A_{n}|^{p}}{\mu_{n}^{p-1}}-V_{n}(\mu_{n-m}-\mu_{n})^{p-1}\frac{|A_{n}|^{p}}{\mu_{n}^{p-1}}\Big)\\
=&\displaystyle\sum_{n=m+1}^{\infty}\frac{div_{m}(V(-\nabla_{m}\mu))_{n}}{\mu_{n}^{p-1}}|A_{n}|^{p}.
\end{align*}This completes the proof.
\end{proof}

\begin{Corollary}
As a particular case of Proposition \ref{p-weightedgeneral}, we obtain the improved weighted dual Hardy inequality established in~\cite{DASMANTON}. Indeed, by taking  $m=1$, $p=2$,  $V_{n}=(n-1)^{\alpha}$, $\mu_{n}=n^{\frac{1-\alpha}{2}}$, where $\alpha\in(1,2]$,  the Proposition  \ref{p-weightedgeneral} gives
\begin{align*}
\displaystyle\sum_{n=2}^{\infty}(n-1)^{\alpha}|A_{n}-A_{n-1}|^{2}&\geq\displaystyle\sum_{n=2}^{\infty}\beta_{n}(\alpha)|A_{n}|^{2}, ~~A_0=A_1=0,
\end{align*}
where $\alpha\in(1, 2]$, and the improved weight sequence  $\beta_{n}(\alpha)$ is defined as below:
\begin{align*}
\beta_{n}(\alpha)& =n^{\alpha}\Big[1+\Big(1-\frac{1}{n}\Big)^{\alpha}-\Big(1-\frac{1}{n}\Big)^{\frac{1+\alpha}{2}}
-\Big(1+\frac{1}{n}\Big)^{\frac{1-\alpha}{2}}\Big], ~n\geq 2.
\end{align*}. In particular, for $\alpha=2$, we get the improved dual Hardy inequality established in~\cite{DASMAN3}.
\end{Corollary}
\section{Improvements of Discrete Dual $p$-Hardy Inequalities}
In this section, we prove that discrete dual $p$ Hardy inequalities (\ref{VHARDYIP2}) and its power-weighted version (\ref{DHI}) admit further improvements for every $p>1$. we first begin with the following theorem, which gives improvement of (\ref{VHARDYIP2}).
\begin{Theorem}{\label{Ptheoremdual}}
Let $p>1$ and $A\in C_c(\mathbb{N}_{0})$ with assumptions $A_{0}=A_{1}=0$, then the following inequality holds for all $n\geq2$, $n\in\mathbb{N}$.
\begin{align}{\label{IMICO}}
&\displaystyle\sum_{n=2}^{\infty}(n-1)^{p}|\nabla A_{n}|^{p}\geq\displaystyle\sum_{n=2}^{\infty}\hat{w_{n}}|A_{n}|^{p},
\end{align}
where the sequence $\hat{w}_{n}$ is defined as below:
\begin{align*}
&\hat{ w}_{n}=n^{p}\big(\frac{1}{pn+1}\big)^{p-1}-(n-1)^{p}\big(\frac{1}{pn-p}\big)^{p-1}>\frac{1}{p^p},~~n\geq2.
\end{align*}
\end{Theorem}
\begin{proof}
 Let $p>1$ . We introduce  the sequences as follows.
 \begin{align}
V_{n} &=
 \left\{
\begin{array}{lll}
    0 & \quad \mbox{if~~} n=1,\\
    %\frac {2\log N-\sqrt\lambda_{n}\log n}{\log N} & \quad \mbox{if~~} N\leq n\leq N^{2}\\
    (n-1)^{p} & \quad \mbox{if~~} n\geq2,
\end{array}
\right.
&
\mu_{n} &=
\left\{
\begin{array}{lll}
    0 & \quad \mbox{if~~} n=0,\\
    %\frac {2\log N-\sqrt\lambda_{n}\log n}{\log N} & \quad \mbox{if~~} N\leq n\leq N^{2}\\
   \frac{\Gamma{ (n)}}{\Gamma \big(n+\frac{1}{p}\big)}  & \quad \mbox{if~~} n\geq1,
\end{array}
\right.
\end{align}where $\Gamma$ is the Euler gamma function.  We see $\mu_{n}$ is decreasing, as $\mu_{n-1}-\mu_{n}=\frac{\Gamma(n)}{p(n-1)\Gamma(n+\frac{1}{p})}>0$ for all $n\geq2$. Therefore, the proposition \ref{p-weightedgeneral} can be applied for this choice of $\mu_{n}$ and $V_{n}$. So, we obtain the inequality for $m=1$ as below:
\begin{align*}
&\displaystyle\sum_{n=2}^{\infty}(n-1)^{p}|\nabla A_{n}|^{p}\geq\displaystyle\sum_{n=2}^{\infty}\hat{ w}_{n}|A_{n}|^{p},
\end{align*} where the improved weight sequence is given by
\begin{align*}
&\hat{ w}_{n}=n^{p}\big(\frac{1}{pn+1}\big)^{p-1}-(n-1)^{p}\big(\frac{1}{pn-p}\big)^{p-1},~~n\geq2.
\end{align*}
To complete the proof of the proposition, it remains to show $\hat {w}_{n}>\frac{1}{p^{p}}$, for all $n\geq2$. Since
$\hat {w}_{n}=\frac{n}{p^{p-1}}G(\frac{1}{n})$, where $G(\frac{1}{n})=(1+\frac{1}{pn})^{1-p}-(1-\frac{1}{n})$, it is enough to show $G(\frac{1}{n})>\frac{1}{pn}$. Let us consider a function $G(x)=(1+\frac{x}{p})^{(1-p)}-(1-x)$ on $x\in(0,\frac{1}{2}]$. It is observed that $G''(x)>0$ on $x\in(0,\frac{1}{2}]$, hence $G(x)$ is convex in $(0,\frac{1}{2}]$. Therefore, one gets $G(x)>G(0)+xG'(0)$. It is also observed that $G(0)=0$ and $G'(0)=\frac{1}{p}$ which implies that $G(x)>\frac{x}{p}$ on $(0,\frac{1}{2}]$. This completes the proof.
\end{proof}

 Motivated by the above construction of suitable $\mu_{n}$ in Theorem \ref{Ptheoremdual}, we next show that the power-weighted discrete dual $p$-Hardy inequality (\ref{DHI}) can be further improved for every $p>1$. To prove this, we first begin with the following lemma.
\begin{Lemma}{\label{Lemmastrictly}}
Let $p>1$. Then, for every $p-1<\alpha\leq p$ and $n\geq2$, $n\in\mathbb{N}$, the following strict inequality holds:
\begin{align*}
\hat{w(\alpha)}_{n}&>\Big(\frac{\alpha+1-p}{p}\Big)^{p}\frac{1}{n^{p-\alpha}},
\end{align*}where $\hat{w}(\alpha)_{n}$ is defined as below:
\begin{align*}
\hat{w(\alpha)}_{n}=n^{\alpha}\big(\frac{\alpha+1-p}{pn+\alpha+1-p}\big)^{p-1}-(n-1)^{\alpha}\big(\frac{\alpha+1-p}{pn-p}\big)^{p-1}.
\end{align*}
\end{Lemma}
\begin{proof}
We can express $\hat{w(\alpha)}_{n}$ in the following form:
\begin{align*}
\hat{w(\alpha)}_{n}&=\Big(\frac{\alpha+1-p}{p}\Big)^{p-1}n^{(\alpha+1-p)}T(n),
\end{align*}where
\begin{align*}
T(n)&=\big(1+\frac{\alpha+1-p}{pn}\big)^{1-p}-(1-\frac{1}{n})^{(\alpha+1-p)}.
\end{align*}To establish this lemma, it suffices to show that $T(n)>\Big(\frac{\alpha+1-p}{p}\Big)\frac{1}{n}$, $n\geq2$, $n\in\mathbb{N}$.
Consider the real-valued function on the interval $(0,\frac{1}{2}]$ given by
\begin{align*}
R(x)&=\big(1+\frac{\alpha+1-p}{p}x\big)^{1-p}-(1-x)^{(\alpha+1-p)}-\Big(\frac{\alpha+1-p}{p}\Big)x.
\end{align*}By successively differentiating $R(x)$ with respect to $x$, we have
\begin{align*}
R'(x)&=(1-p)\Big(\frac{\alpha+1-p}{p}\Big)\big(1+\frac{\alpha+1-p}{p}x\big)^{-p}+(\alpha+1-p)(1-x)^{(\alpha-p)}-\Big(\frac{\alpha+1-p}{p}\Big),
\end{align*}and
\begin{align*}
R''(x)&=p(p-1)\Big(\frac{\alpha+1-p}{p}\Big)^{2}\big(1+\frac{\alpha+1-p}{p}x\big)^{-p-1}+(p-\alpha)(\alpha+1-p)(1-x)^{(\alpha-p-1)}.
\end{align*} As $\alpha\in(p-1,p]$, we have $(\alpha+1-p)>0$ and $(p-\alpha)\geq0$. Consequently, $R''(x)>0$ for all $x\in(0,\frac{1}{2}]$. It is observed that $R(0)=R'(0)=0$. Hence $R(x)>0$ for $x\in(0,\frac{1}{2}]$. This shows that $R(\frac{1}{n})=\big(T(n)-\Big(\frac{\alpha+1-p}{p}\Big)\frac{1}{n}\big)>0$ which establishes the desired inequality.
\end{proof}
\begin{Theorem}
Let $p>1$ and $p-1<\alpha\leq p$. Suppose that $A\in C_c(\mathbb{N}_{0})$ satisfies $A_0=A_1=0$. Then, for every $n\geq2$, $n\in\mathbb{N}$, the following inequality holds:
\begin{align}{\label{GDIMITheorem}}
&\displaystyle\sum_{n=2}^{\infty}(n-1)^{\alpha}|\nabla A_{n}|^{p}\geq\displaystyle\sum_{n=2}^{\infty}\hat {w(\alpha)}_{n}|A_{n}|^{p},
\end{align}where
\begin{align*}
&\hat{w(\alpha)}_{n}=n^{\alpha}\big(\frac{\alpha+1-p}{pn+\alpha+1-p}\big)^{p-1}-(n-1)^{\alpha}\big(\frac{\alpha+1-p}{pn-p}\big)^{p-1}>\Big(\frac{\alpha+1-p}{p}\Big)^{p}\frac{1}{n^{p-\alpha}},~~n\geq2.
\end{align*}
\end{Theorem}
\begin{proof}
Let $p>1$ and  $p-1<\alpha\leq p$, then we now define the sequences as follows:\\
 \begin{align}
V_{n} &=
 \left\{
\begin{array}{lll}
    0 & \quad \mbox{if~~} n=1,\\
    %\frac {2\log N-\sqrt\lambda_{n}\log n}{\log N} & \quad \mbox{if~~} N\leq n\leq N^{2}\\
    (n-1)^{\alpha} & \quad \mbox{if~~} n\geq2,
\end{array}
\right.
&
\mu_{n} &=
\left\{
\begin{array}{lll}
    0 & \quad \mbox{if~~} n=0,\\
    %\frac {2\log N-\sqrt\lambda_{n}\log n}{\log N} & \quad \mbox{if~~} N\leq n\leq N^{2}\\
   \frac{\Gamma{ (n)}}{\Gamma \big(n+\frac{1+\alpha-p}{p}\big)}  & \quad \mbox{if~~} n\geq1,
\end{array}
\right.
\end{align}where $\Gamma$ is the Euler gamma function. Since $\mu_{n-1}-\mu_{n}=\frac{(1+\alpha-p)\Gamma(n)}{p(n-1)\Gamma(n-1+\frac{1+\alpha-p}{p})}>0$ for all $n\geq2$, hence $\mu_{n}$ is decreasing. Hence, Proposition \ref{p-weightedgeneral} can be applied to the above choice of $\mu_n$ and $V_{n}$ with $m=1$.  Consequently, we obtain the desired improved weighted dual $p$- Hardy inequality. The inequality $\hat{w(\alpha)}_{n}>\Big(\frac{\alpha+1-p}{p}\Big)^{p}\frac{1}{n^{p-\alpha}}$ follows directly from Lemma \ref{Lemmastrictly}. This completes the proof of the theorem.
\end{proof}

\section{ Discrete $p$-Birman Inequality with power weight}
In this section, we prove a weighted discrete $p$-Birman inequality and, as a consequence, obtain a weighted discrete Rellich inequality as a special case.
We first recall the notation required for the proof of the weighted $p$-Birman inequality. For $l\in\mathbb N$, we recall the Pochhammer symbol
$(a)_{l}:=\prod_{j=0}^{l-1}(a+j)$. In addition, we shall use the forward shift operator $S$, acting on a sequence $u_{n}$ as
$Su_{n}=u_{n+1}$, and discrete gradient operator $\nabla u_{n}=u_{n}-u_{n-1}$, $n\in\mathbb{N}_0$. We shall employ these notations in the proof of the following theorem \ref{WBI}, which establishes  $p$-Birman inequality with power weights.
\begin{Theorem}{\label{WBI}}
Let $p>1$, $\alpha\in[0,p-1)$, and $A\in C_c(\mathbb{N}_0)$ satisfy  $A_{n}=0$ for $n<l$, where $l\in\mathbb{N}$, then for all $n\geq l$, $n\in\mathbb{N}_0$, the following inequality holds.
\begin{align}{\label{WBM}}
&\displaystyle\sum_{n=l}^{\infty}(n-l+1)^{\alpha}|\nabla^{l} A_{n}|^{p}\geq\Big(\frac{p-\alpha-1}{p}\Big)_{l}^{p}\displaystyle\sum_{n=l}^{\infty}n^{\alpha-lp}|A_{n}|^{p}.
\end{align}
\end{Theorem}
\begin{proof}
We prove the theorem by induction on $l$, $l\in\mathbb{N}$. For case $l=1$, the inequality (\ref{WBM}) is the the power-type weighted Hardy inequality (\ref{COPIPAR}). We now turn to the case $l=2$. Let us assume that $A_{1}=A_{0}=0$. Then we have,
\begin{align*}
\displaystyle\sum_{n=2}^{\infty}(n-1)^{\alpha}|\nabla^{2} A_{n}|^{p}=&\displaystyle\sum_{n=1}^{\infty}n^{\alpha}|\nabla(\nabla A_{n+1})|^{p}\\
%=&\displaystyle\sum_{n=1}^{\infty}n^{\alpha}|\nabla(\nabla S A_{n})|^{p}\\
\geq&\Big(\frac{p-\alpha-1}{p}\Big)^{p}\displaystyle\sum_{n=1}^{\infty}n^{\alpha-p}|\nabla A_{n+1}|^{p}~~[~ by~using~inequality~(\ref{COPIPAR})~ ]\\
=&\Big(\frac{p-\alpha-1}{p}\Big)^{p}\displaystyle\sum_{n=1}^{\infty}n^{\hat{\alpha}}|\nabla S A_{n})|^{p}\\
&[Since,~ \hat{\alpha}=(\alpha-p)< 0,~we~apply~~ inequality~ (\ref{ECOPIPAR})]\\
\geq&\Big(\frac{p-\alpha-1}{p}\Big)^{p}\Big(\frac{p-\hat{\alpha}-1}{p}\Big)^{p}\displaystyle\sum_{n=1}^{\infty}(n+1)^{\hat{\alpha}}|SA_{n}|^{p}\\
=&\Big(\frac{p-1-\alpha}{p}\Big)_{2}^{p}\displaystyle\sum_{n=2}^{\infty}n^{\alpha-2p}| A_{n}|^{p}
\end{align*}Suppose the result holds for $l-1$. We shall establish its validity for $l$, assuming that $A_{n}=0$ for every $n<l$. Applying the induction hypothesis, we get
\begin{align*}
\displaystyle\sum_{n=l}^{\infty}(n-l+1)^{\alpha}|\nabla^{l} A_{n}|^{p}=&\displaystyle\sum_{n=l-1}^{\infty}(n-l+2)^{\alpha}|\nabla^{l-1}(\nabla A_{n+1})|^{p}\\
&\geq\Big(\frac{p-\alpha-1}{p}\Big)_{l-1}^{p}\displaystyle\sum_{n=l-1}^{\infty}n^{\alpha-(l-1)p}|\nabla A_{n+1}|^{p}\\
&=\Big(\frac{p-\alpha-1}{p}\Big)_{l-1}^{p}\displaystyle\sum_{n=1}^{\infty}n^{\tilde{\alpha}}|\nabla SA_{n}|^{p}\\
&[~Since~, \tilde{\alpha}=(\alpha-(l-1)p)< 0,~we~apply~~ inequality~ (\ref{ECOPIPAR})~].\\
&\geq\Big(\frac{p-\alpha-1}{p}\Big)_{l-1}^{p}\Big(\frac{p-\tilde{\alpha}-1}{p}\Big)^{p}\displaystyle\sum_{n=1}^{\infty}(n+1)^{\tilde{\alpha}-p}| SA_{n}|^{p}\\
&=\Big(\frac{p-\alpha-1}{p}\Big)_{l}^{p}\displaystyle\sum_{n=l}^{\infty}n^{\alpha-lp}| A_{n}|^{p}
\end{align*}
Therefore, the desired inequality holds for $l$. By induction, the theorem follows.
\end{proof}

\begin{Corollary}
If we put $p=2$ and $l=2$ in the above inequality (\ref{WBM}), then for $\alpha\in[0,2)$, we obtain the following Rellich inequality with power weights.
\begin{align}{\label{WRellich}}
&\displaystyle\sum_{n=1}^{\infty}n^{\alpha}|\Delta A_{n}|^{2}\geq\frac{(\alpha^{2}-4\alpha+3)^{2}}{16}\displaystyle\sum_{n=1}^{\infty}\frac{|A_{n}|^{2}}{n^{4}}n^{\alpha},~~A_{0}=A_{1}=0,
\end{align} Here, the discrete Dirichlet Laplacian $\Delta$ on $A_n$, introduced in \cite{BGDKFS}, is defined by
\begin{align*}
\Delta A_{n} &= \left\{
\begin{array}{lll}
    2A_{0}-A_{1} & \quad \mbox{if~~} n=0,\\
    2A_{n}-A_{n-1}-A_{n+1} & \quad \mbox{if~~} n\in\mathbb{N}.
\end{array}\right.
\end{align*} For $\alpha=0$ in (\ref{WRellich}), we get discrete Rellich inequality \cite{BGDKFS}.
\end{Corollary}

\section{ A Multivariate discrete p-dual Hardy inequality}
In this section, we establish the $p$-dual Hardy inequality (\ref{VHARDYI}) in the multivariable setting.
The main result is stated in the following theorem, where the sharpness of the constant is also established.
\begin{thm}{\label{r-thm}}
Let $\{a_{m_{1}m_{2}\cdots m_{k}}\}$ be a $k$-fold sequence of complex numbers. Then
\begin{align}{\label{r-I}}
&\displaystyle\sum_{m_{1}=1}^{\infty}\displaystyle\sum_{m_{2}=1}^{\infty}\cdots\displaystyle\sum_{m_{k}=1}^{\infty}\Big|\displaystyle\sum_{i_{1}=m_{1}}^{\infty}\displaystyle\sum_{i_{2}=m_{2}}^{\infty}\cdots\displaystyle\sum_{i_{k}=m_{k}}^{\infty}a_{i_{1}i_{2}\cdots i_{k}}\Big|^{p}\nonumber\\
&\leq(p^{p})^{k}\displaystyle\sum_{m_{1}=1}^{\infty}\displaystyle\sum_{m_{2}=1}^{\infty}\cdots\displaystyle\sum_{m_{k}=1}^{\infty}\Big(m_{1}m_{2}\cdots m_{k}\Big)^{p}\mid a_{m_{1}m_{2}\cdots m_{k}}\mid^{p},
\end{align}where the associated constant term $(p^{p})^{k}$ is sharp.
\end{thm}
\begin{proof}
We establish the theorem by induction in $k$, $k\in\mathbb{N}$. For $k=1$, the inequality (\ref{r-I}) gives classical discrete dual $p$-Hardy inequality (\ref{VHARDYI}). To established the inequality (\ref{r-I}) for $k=2$, we define $A_{i_{1}m_{2}}=\displaystyle\sum_{i_{2}=m_{2}}^{\infty}a_{i_{1}i_{2}}$ and begin with
\begin{align*}
\displaystyle\sum_{m_{1}=1}^{\infty}\displaystyle\sum_{m_{2}=1}^{\infty}\Big|\displaystyle\sum_{i_{1}=m_{1}}^{\infty}\displaystyle\sum_{i_{2}=m_{2}}^{\infty}a_{i_{1}i_{2}}\Big|^{p}
=&\displaystyle\sum_{m_{1}=1}^{\infty}\displaystyle\sum_{m_{2}=1}^{\infty}\Big|\displaystyle\sum_{i_{1}=m_{1}}^{\infty}A_{i_{1}m_{2}}\Big|^{p}\leq p^{p}\displaystyle\sum_{m_{2}=1}^{\infty}\displaystyle\sum_{m_{1}=1}^{\infty}m_{1}^{p}|A_{m_{1}m_{2}}|,
%=&\displaystyle\sum_{n=1}^{\infty}\displaystyle\sum_{m_{1}=1}^{\infty}\Big|\displaystyle\sum_{i=m_{1}}^{\infty}A_{im_{2}}\Big|^{p}\leq p^{p}\displaystyle\sum_{m_{2}=1}^{\infty}\displaystyle\sum_{m_{1}=1}^{\infty}m_{1}^{p}|A_{m_{1}m_{2}}|
\end{align*}where the last inequality is an immediate consequence $\displaystyle\sum_{m_{1}=1}^{\infty}\Big|\displaystyle\sum_{i_{1}=m_{1}}^{\infty}A_{i_{1}m_{2}}\Big|^{p}\leq p^{p}\displaystyle\sum_{m_{1}=1}^{\infty}m_{1}^{p}|A_{m_{1}m_{2}}|^{p}$.\\Hence, we have
\begin{align*}
\displaystyle\sum_{m_{1}=1}^{\infty}\displaystyle\sum_{m_{2}=1}^{\infty}\Big|\displaystyle\sum_{i=m_{1}}^{\infty}\displaystyle\sum_{i_{2}=m_{2}}^{\infty}a_{i_{1}i_{2}}\Big|^{p}\leq& p^{p}\displaystyle\sum_{m_{2}=1}^{\infty}\displaystyle\sum_{m_{1}=1}^{\infty}m_{1}^{p}|\displaystyle\sum_{i_{2}=m_{2}}^{\infty}a_{m_{1}i_{2}}|^{p}\\
=&p^{p}\displaystyle\sum_{m_{1}=1}^{\infty}m_{1}^{p}\Big(\displaystyle\sum_{m_{2}=1}^{\infty}|\displaystyle\sum_{i_{2}=m_{2}}^{\infty}a_{m_{1}i_{2}}|^{p}\Big)\\
\leq& (p^{p})^{2}\displaystyle\sum_{m_{1}=1}^{\infty}m_{1}^{p}\displaystyle\sum_{m_{2}=1}^{\infty}m_{2}^{p}|a_{m_{1}m_{2}}|^{p}\\
=&(p^{p})^{2}\displaystyle\sum_{m_{1}=1}^{\infty}\displaystyle\sum_{m_{2}=1}^{\infty}(m_{1}m_{2})^{p}|a_{m_{1}m_{2}}|^{p}
\end{align*}
and therefore
\begin{align*}
&\displaystyle\sum_{m_{1}=1}^{\infty}\displaystyle\sum_{m_{2}=1}^{\infty}\Big|\displaystyle\sum_{i=m_{1}}^{\infty}\displaystyle\sum_{j=m_{2}}^{\infty}a_{ij}\Big|^{p}
\leq (p^{p})^{2}\displaystyle\sum_{m_{1}=1}^{\infty}\displaystyle\sum_{m_{2}=1}^{\infty}(m_{1}m_{2})^{p}|a_{m_{1}m_{2}}|^{p}.
\end{align*} Suppose that the inequality (\ref{r-I}) holds for $k$-fold series. We have to show it is true for $(k+1)$-fold series. We introduce
\begin{align}{\label{D define}}
&R_{i_{2}\ldots i_{k}i_{k+1}}=\displaystyle\sum_{i_{1}=m_{1}}^{\infty}a_{i_{1}i_{2}\ldots i_{k}i_{k+1}}
\end{align}and
Using (\ref{D define}), we have
\begin{align*}
&\displaystyle\sum_{m_{1}=1}^{\infty}\displaystyle\sum_{m_{2}=1}^{\infty}\ldots\displaystyle\sum_{m_{k}=1}^{\infty}\displaystyle\sum_{m_{k+1}=1}^{\infty}\Big|\displaystyle\sum_{i_{1}=m_{1}}^{\infty}\displaystyle\sum_{i_{2}=m_{2}}^{\infty}\ldots\displaystyle\sum_{i_{k}=m_{k}}^{\infty}\displaystyle
\sum_{i_{k+1}=m_{k+1}}^{\infty}a_{i_{1}i_{2}\ldots i_{k}i_{k+1}} \Big|^{p}\\
=&\displaystyle\sum_{m_{1}=1}^{\infty}\displaystyle\sum_{m_{2}=1}^{\infty}\ldots\displaystyle\sum_{m_{k}=1}^{\infty}\displaystyle\sum_{m_{k+1}=1}^{\infty}\Big|\displaystyle\sum_{i_{2}=m_{2}}^{\infty}\ldots\displaystyle\sum_{i_{k}=m_{k}}^{\infty}\displaystyle\sum_{i_{k+1}=m_{k+1}}^{\infty}R_{i_{2}\ldots i_{k}i_{k+1}}\Big|^{p}.
\end{align*}By applying  $k$-fold discrete $p$-dual Hardy inequality , we obtain
\begin{align*}
&\displaystyle\sum_{m_{1}=1}^{\infty}\Bigg(\displaystyle\sum_{m_{2}=1}^{\infty}\ldots\displaystyle\sum_{m_{k}=1}^{\infty}\displaystyle\sum_{m_{k+1}=1}^{\infty}\Big|\displaystyle\sum_{i_{2}=m_{2}}^{\infty}\ldots\displaystyle\sum_{i_{k}=m_{k}}^{\infty}\displaystyle\sum_{i_{k+1}=m_{k+1}}^{\infty}R_{i_{2}\ldots i_{k}i_{k+1}}\Big|^{p}\Bigg)\\
\leq&(p^{p})^{k}\displaystyle\sum_{m_{1}=1}^{\infty}\displaystyle\sum_{m_{2}=1}^{\infty}\ldots\displaystyle\sum_{m_{k}=1}^{\infty}\displaystyle\sum_{m_{k+1}=1}^{\infty}(m_{2}m_{3}\cdots m_{k+1})^{p}\mid R_{m_{2}m_{3}\ldots m_{k}m_{k+1}}\mid^{p}\\
=&(p^{p})^{k}\displaystyle\sum_{m_{1}=1}^{\infty}\displaystyle\sum_{m_{2}=1}^{\infty}\ldots\displaystyle\sum_{m_{k}=1}^{\infty}\displaystyle\sum_{m_{k+1}=1}^{\infty}(m_{2}m_{3}\cdots m_{k+1})^{p}
\mid\displaystyle\sum_{i_{1}=m_{1}}^{\infty}a_{i_{1}m_{2}\cdots m_{k}m_{k+1}}\mid^{p}(\mbox{by applying}~(\ref{D define})).
\end{align*}
It follows from Fubini's theorem that
\begin{align*}
&\displaystyle\sum_{m_{1}=1}^{\infty}\displaystyle\sum_{m_{2}=1}^{\infty}\ldots\displaystyle\sum_{m_{k}=1}^{\infty}\displaystyle\sum_{m_{k+1}=1}^{\infty}\Big|\displaystyle\sum_{i_{2}=m_{2}}^{\infty}\ldots\displaystyle\sum_{i_{k}=m_{k}}^{\infty}\displaystyle\sum_{i_{k+1}=m_{k+1}}^{\infty}R_{i_{2}\ldots i_{k}i_{k+1}}\Big|^{p}.\\
&\leq(p^{p})^{k}\displaystyle\sum_{m_{2}=1}^{\infty}\ldots\displaystyle\sum_{m_{k}=1}^{\infty}\displaystyle\sum_{m_{k+1}=1}^{\infty}(m_{2}m_{3}\cdots m_{k+1})^{p}\Bigg(\displaystyle\sum_{m_{1}=1}^{\infty}\mid\displaystyle\sum_{i_{1}=m_{1}}^{\infty}a_{i_{1}m_{2}\cdots m_{k}m_{k+1}}\mid^{p}\Bigg)
\end{align*}
\begin{align*}
&\leq(p^p)^{k+1}\displaystyle\sum_{m_{2}=1}^{\infty}\ldots\displaystyle\sum_{m_{k}=1}^{\infty}\displaystyle\sum_{m_{k+1}=1}^{\infty}(m_{2}m_{3}\cdots m_{k+1})^{p}\displaystyle\sum_{m_{1}=1}^{\infty}m_{1}^{p}\mid a_{m_{1}m_{2}\ldots m_{k}m_{k+1}}\mid^{p}\\
&(\mbox{by applying the 1-fold dual $p$-Hardy inequality})\\
=&(p^p)^{k+1}\displaystyle\sum_{m_{1}=1}^{\infty}\displaystyle\sum_{m_{2}=1}^{\infty}\ldots\displaystyle\sum_{m_{k}=1}^{\infty}\displaystyle\sum_{m_{r+1}=1}^{\infty}(m_{1}m_{2}\cdots m_{k}m_{k+1})^{p}\mid a_{m_{1}m_{2}\ldots m_{k}m_{k+1}}\mid^{p}.
\end{align*}
Thus, Theorem (\ref{r-thm}) follows for $(k+1)$-fold series. By the principle of mathematical induction, we conclude that the stated result holds for all $k\in\mathbb{N}$. Hence, the proof is complete.\\
 To prove the sharpness of the constant, we choose $a_{i_{1}i_{2}\ldots i_{k}}=\prod_{j=1}^{k}i_{j}^{-\frac{1}{p}-1-\epsilon}$, where $\epsilon$ is arbitrary small such that $\epsilon>0$. Now using this test sequence we have,
 \begin{align*}
 &\displaystyle\sum_{m_{1}=1}^{\infty}\displaystyle\sum_{m_{2}=1}^{\infty}\cdots\displaystyle\sum_{m_{k}=1}^{\infty}\Big|\displaystyle\sum_{i_{1}=m_{1}}^{\infty}\displaystyle\sum_{i_{2}=m_{2}}^{\infty}\cdots\displaystyle\sum_{i_{k}=m_{k}}^{\infty}\prod_{j=1}^{k}i_{j}^{-\frac{1}{p}-1-\epsilon}\Big|^{p}\\
 =&\prod_{j=1}^{k}\Big(\displaystyle\sum_{m_{j}=1}^{\infty}\Big|\displaystyle\sum_{i_{j}=m_{j}}^{\infty}i_{j}^{-\frac{1}{p}-1-\epsilon}\Big|^{p}\Big)\\
 >&\prod_{j=1}^{k}\Big(\displaystyle\sum_{m_{j}=1}^{\infty}\Big|\frac{m_{j}^{-\frac{1}{p}-\epsilon}}{\frac{1}{p}+\epsilon}\Big|^{p}\Big)~~[~Since,~\displaystyle\sum_{i_{j}=m_{j}}^{\infty}i_{j}^{-\frac{1}{p}-1-\epsilon}>\int_{m_{j}}^{\infty} x^{-\frac{1}{p}-1-\epsilon}~ dx~]\\
 =&\Big(\frac{p^{p}}{(1+p\epsilon)^{p}}\Big)^{k}\prod_{j=1}^{k}\Big(\displaystyle\sum_{m_{j}=1}^{\infty}\Big|m_{j}^{-1-p\epsilon}\Big|\Big)\\
 =&\Big(\frac{p^{p}}{(1+p\epsilon)^{p}}\Big)^{k}\displaystyle\sum_{m_{1}=1}^{\infty}\displaystyle\sum_{m_{2}=1}^{\infty}\cdots\displaystyle\sum_{m_{k}=1}^{\infty}\Big(m_{1}m_{2}\cdots m_{k}\Big)^{p}\Big|\prod_{j=1}^{k}m_{j}^{-\frac{1}{p}-1-\epsilon}\Big|^{p}.
 \end{align*} As $\epsilon\rightarrow0$, it follows that the constant is sharp. This completes the proof of the theorem.

\end{proof}

\end{document}